\documentclass[11pt]{article}

\usepackage[T1]{fontenc}
\usepackage[utf8]{inputenc}
\usepackage{newtxtext}
\usepackage{amsmath,amssymb,amsthm,mathtools}
\usepackage{newtxmath}
\usepackage{geometry}
\usepackage{microtype}
\usepackage[dvipsnames]{xcolor}
\usepackage{fancyhdr}
\usepackage[super,sort&compress]{natbib}
\usepackage[colorlinks=true,linkcolor=NavyBlue,citecolor=NavyBlue,urlcolor=NavyBlue]{hyperref}
\hypersetup{
  pdftitle={Calabi metric on the Space of k-Hessian Potentials},
  pdfauthor={Author Name},
  pdfsubject={Complex Hessian equations and Fisher--Rao geometry},
  bookmarksopen=true,
  bookmarksnumbered=true
}

\allowdisplaybreaks
\numberwithin{equation}{section}

\newtheorem{theorem}{Theorem}[section]
\newtheorem{proposition}[theorem]{Proposition}

\newtheorem{corollary}[theorem]{Corollary}

\theoremstyle{definition}
\newtheorem{definition}[theorem]{Definition}

\theoremstyle{remark}
\newtheorem{remark}[theorem]{Remark}

\newcommand{\ii}{\sqrt{-1}}
\newcommand{\ddc}{\ii\,\partial\bar\partial}
\newcommand{\R}{\mathbb{R}}
\newcommand{\Hk}{\mathcal{H}_k(X,\omega)}
\newcommand{\Hkbar}{\widetilde{\mathcal{H}}_k(X,\omega)}
\newcommand{\DV}{\mathcal{D}_V}
\newcommand{\Tk}{\mathcal{T}_k}
\newcommand{\Qk}{\mathcal{Q}_k}
\newcommand{\FR}{\mathrm{FR}}
\newcommand{\HM}{\mathrm{HM}}
\newcommand{\HC}{\mathrm{HC}}
\newcommand{\Cinf}{C^{\infty}}
\DeclareMathOperator{\tr}{tr}

\title{Calabi metric on the Space of $k$-Hessian Potentials}
\author{Linwei Zhang}
\date{}

\begin{document}

\maketitle
\thispagestyle{plain}
\vspace{-1.2em}

\begin{abstract}
Let $(X,\omega)$ be a compact K\"ahler manifold of complex dimension $n$,
and let $1\leq k\leq n$. We consider the space of smooth $k$-admissible
potentials introduced in the study of the Hessian-cscK equation by
Guo--Smith--Tong. The complex $k$-Hessian equation identifies this space,
modulo constants, with the space of smooth positive densities of fixed total
mass. We compute the differential of the Hessian density map and use it to
pull back the Fisher--Rao metric. The resulting Riemannian metric is
\[
 G^{\HC,k}_{\varphi}(u,v)
 =\int_X
 \bigl(\Delta_{\varphi}^{(k)}u\bigr)
 \bigl(\Delta_{\varphi}^{(k)}v\bigr)
 \,\omega_{\varphi}^{k}\wedge\omega^{n-k}.
\]
We then show that the square-root Hessian density map
\[
 \varphi\longmapsto
 2\sqrt{\frac{\omega_{\varphi}^{k}\wedge\omega^{n-k}}{\omega^n}}
\]
is a global Riemannian isometry onto the positive part of the
$L^2(X,\omega^n)$-sphere of radius $2\sqrt{V}$, where
$V=\int_X\omega^n$. For $k=n$, this recovers the spherical description of
the classical Calabi metric. We also derive the curvature,
geodesic and diameter formulas, and prove a quantitative local
comparison between adjacent Hessian--Fisher distances.
\end{abstract}

\setcounter{tocdepth}{2}
\tableofcontents
\clearpage

\section{Introduction}

Let $(X,\omega)$ be a compact K\"ahler manifold of complex dimension
$n$. The space of K\"ahler potentials
\[
 \mathcal H_n(X,\omega)
 :=\bigl\{\varphi\in C^\infty(X,\mathbb R):
 \omega_\varphi:=\omega+\sqrt{-1}\,\partial\bar\partial\varphi>0\bigr\}
\]
carries a rich collection of infinite-dimensional Riemannian structures.
Among them, the Mabuchi $L^2$ metric plays a central role in the study of
constant scalar curvature K\"ahler metrics. Another natural Riemannian
structure is the Calabi metric.

In \cite{Calamai2012}, Calamai studied the geometry of the Calabi metric
on the space of K\"ahler metrics in a fixed K\"ahler class. By Yau's
solution of the Calabi conjecture \cite{Yau1978}, the complex
Monge--Amp\`ere equation identifies K\"ahler potentials modulo constants
with smooth positive volume densities of prescribed total mass.
Using this identification, Calamai showed that the Calabi metric admits
a spherical realization: after taking the square root of the volume
density, the space is isometric to the positive part of an
infinite-dimensional $L^2$-sphere. This description gives constant
positive sectional curvature, explicit geodesics and solutions of the
associated geodesic initial- and boundary-value problems, as well as
information on distance and diameter.

More recently, Guo--Smith--Tong \cite{GST2023} introduced an analogue of
the Mabuchi metric on the space $\mathcal H_k(X,\omega)$ of
$k$-Hessian-admissible potentials, for $1\leq k\leq n$. They studied the
resulting infinite-dimensional Riemannian geometry and its connection
with the Hessian-cscK equations. Their construction is a $k$-Hessian
extension of the Mabuchi geometry, rather than of the Calabi geometry.

In this paper, using the solvability of the complex $k$-Hessian
equation, we define a new Riemannian metric on
$\mathcal H_k(X,\omega)/\mathbb R$ by pulling back the Fisher--Rao metric
on the space of smooth positive densities. For $k=n$, it reduces to the
classical Calabi metric. We establish a global spherical model for this
metric and use it to compute its sectional curvature, describe its
geodesics, and determine the diameter of the potential space. Finally,
we compare the distances associated with different Hessian orders in a
neighborhood of the origin of $\mathcal H_n(X,\omega)$, obtaining a
quantitative local comparison and strict monotonicity for adjacent
orders. Our main theorem is as follows:

\begin{theorem}[Main result]\label{thm:introduction-main}
Let $(X,\omega)$ be a compact connected K\"ahler manifold of complex
 dimension $n$, and let $1\leq k\leq n$. Denote by $\Hk$ the space of
 smooth $k$-Hessian-admissible potentials. Then the quotient
 $\Hkbar=\Hk/\mathbb R$ admits a (weak) Riemannian metric $g_k$
 with constant sectional curvature
 \[
  \operatorname{Sec}_{g_k}\equiv\frac{1}{4V},
 \]
 and its induced geodesic distance $d_{g_k}$ has diameter
 \[
  \operatorname{diam}(\Hkbar,d_{g_k})=\pi\sqrt V,
  \qquad V:=\int_X\omega^n.
 \]
\end{theorem}
 {\bf Acknowledgements}:
The authors  would like to thank Professors Haozhao Li for his long-time support and many useful suggestions.\\
  {\bf On the usage of AI}:
The ideas presented in this paper are entirely due to the author. ChatGPT Plus was used to simplify the proof of Theorem~\ref{thm:adjacent-distance-comparison}. The authors take full responsibility for the contents of the paper.

\section{Background}

Let $(X,\omega)$ be a compact connected K\"ahler manifold of complex
dimension $n$, and fix an integer $1\leq k\leq n$. For
$\varphi\in\Cinf(X,\R)$ we write
\begin{equation}\label{eq:omega-phi}
 \omega_{\varphi}=\omega+\ddc\varphi,
 \qquad
 V:=\int_X\omega^n.
\end{equation}
The purpose of this section is to recall the ingredients used below:
the solvability of the complex $k$-Hessian equation, the space of
$k$-Hessian potentials of Guo--Smith--Tong, the Fisher--Rao metric on
positive densities, and the geometry of the associated positive
$L^2$-sphere.

\subsection{The complex \texorpdfstring{$k$}{k}-Hessian equation}

Let $\sigma_j$ denote the $j$-th elementary symmetric polynomial and set
\begin{equation}\label{eq:garding-cone}
 \Gamma_k
 :=\left\{\lambda\in\R^n:
 \sigma_j(\lambda)>0,\quad 1\leq j\leq k\right\}.
\end{equation}
A smooth real $(1,1)$-form $\alpha$ is called \emph{$k$-positive with
respect to $\omega$} if
\[
 \lambda(\omega^{-1}\alpha)\in\Gamma_k.
\]
Equivalently,
\begin{equation}\label{eq:k-positive}
 \alpha^j\wedge\omega^{n-j}>0,
 \qquad 1\leq j\leq k.
\end{equation}

We shall use the non-degenerate complex $k$-Hessian equation
\begin{equation}\label{eq:hessian-equation}
 \omega_{\varphi}^{k}\wedge\omega^{n-k}
 =F\,\omega^n,
 \qquad F>0.
\end{equation}
Integrating \eqref{eq:hessian-equation} gives the necessary compatibility
condition
\begin{equation}\label{eq:mass-condition}
 \int_XF\,\omega^n
 =\int_X\omega_{\varphi}^{k}\wedge\omega^{n-k}
 =\int_X\omega^n
 =V.
\end{equation}
The smooth non-degenerate complex Hessian equation on a compact K\"ahler
manifold is solvable in full generality; see Dinew--Ko\l odziej
\cite{DK2017}, completing the program initiated by Hou--Ma--Wu
\cite{HMW2010}. We record the form needed in this paper.

\begin{theorem}[Complex $k$-Hessian solvability]\label{thm:hessian-solvability}
Let $(X,\omega)$ be a compact K\"ahler manifold. Suppose
$F\in\Cinf(X)$ satisfies
\[
 F>0,
 \qquad
 \int_XF\,\omega^n=V.
\]
Then there exists a smooth $k$-admissible function $\varphi$, unique modulo
constants, such that
\[
 \omega_{\varphi}^{k}\wedge\omega^{n-k}=F\,\omega^n.
\]
\end{theorem}

In the projective setting, Chen--Nie--Xu recently obtained a uniform
Nakai--Moishezon type numerical criterion for complex Hessian type equations
\cite{CNX2026}. In particular, for the pure complex $k$-Hessian equation,
their criterion relates the existence of a smooth $\Gamma_k$-admissible
representative, solvability for normalized smooth positive right-hand sides,
and uniform numerical inequalities on irreducible subvarieties. In the
present paper the class is $[\omega]$ itself; since $\omega$ is K\"ahler, a
smooth $\Gamma_k$-admissible representative is already available. Thus their
result gives the same solvability conclusion in the projective case. We use
Theorem~\ref{thm:hessian-solvability} below so that no projectivity assumption
is required.

\subsection{The space of \texorpdfstring{$k$}{k}-Hessian potentials}

Following Guo--Smith--Tong \cite{GST2023}, define
\begin{equation}\label{eq:Hk-definition}
 \Hk
 :=\left\{\varphi\in\Cinf(X,\R):
 \omega_{\varphi}^{j}\wedge\omega^{n-j}>0,
 \quad 1\leq j\leq k\right\}.
\end{equation}
Equivalently,
\[
 \lambda(\omega^{-1}\omega_{\varphi})\in\Gamma_k.
\]
Since $\Gamma_k$ is open, $\Hk$ is an open subset of
$\Cinf(X,\R)$ and, formally,
\begin{equation}\label{eq:tangent-Hk}
 T_{\varphi}\Hk\simeq\Cinf(X,\R).
\end{equation}

Guo--Smith--Tong introduced on $\Hk$ the Hessian analogue of the Mabuchi
$L^2$ metric,
\begin{equation}\label{eq:GST-metric}
 G^{\HM,k}_{\varphi}(u,v)
 :=\int_Xuv\,\omega_{\varphi}^{k}\wedge\omega^{n-k},
\end{equation}
and studied its geometry in connection with the Hessian-cscK equation
\cite{GST2023}. When $k=n$, \eqref{eq:GST-metric} reduces to the classical
Mabuchi metric.

All the constructions considered below depend only on $\omega_{\varphi}$ and
are therefore invariant under addition of constants. We consequently work on
the quotient
\begin{equation}\label{eq:Hk-quotient}
 \Hkbar:=\Hk/\R,
\end{equation}
with tangent space
\begin{equation}\label{eq:tangent-Hkbar}
 T_{[\varphi]}\Hkbar
 \simeq \Cinf(X,\R)/\R.
\end{equation}

\subsection{The Fisher metric on positive densities}

Let
\begin{equation}\label{eq:density-space}
 \DV
 :=\left\{F\in\Cinf(X):
 F>0,\quad \int_XF\,\omega^n=V\right\}.
\end{equation}
The tangent space at $F\in\DV$ is
\begin{equation}\label{eq:tangent-DV}
 T_F\DV
 =\left\{h\in\Cinf(X):
 \int_Xh\,\omega^n=0\right\}.
\end{equation}
The Fisher metric is
\begin{equation}\label{eq:fisher-rao}
 G_F^{\FR}(h_1,h_2)
 :=\int_X\frac{h_1h_2}{F}\,\omega^n.
\end{equation}

Its spherical form is obtained from the square-root map
\begin{equation}\label{eq:square-root-map}
 \mathcal S:\DV\longrightarrow L^2(X,\omega^n),
 \qquad
 \mathcal S(F):=2\sqrt F.
\end{equation}
Indeed,
\[
 \|\mathcal S(F)\|_{L^2(\omega^n)}^2
 =4\int_XF\,\omega^n
 =4V.
\]
Hence $\mathcal S$ maps $\DV$ onto the positive part
\begin{equation}\label{eq:positive-sphere}
 S_{2\sqrt V}^{+}
 :=\left\{q\in\Cinf(X):
 q>0,\quad \int_Xq^2\omega^n=4V\right\}
\end{equation}
of the $L^2(X,\omega^n)$-sphere of radius $2\sqrt V$. Moreover,
\begin{equation}\label{eq:dS}
 D\mathcal S_F(h)=\frac{h}{\sqrt F},
\end{equation}
and therefore
\begin{equation}\label{eq:FR-sphere-isometry}
 \left\langle D\mathcal S_F(h_1),D\mathcal S_F(h_2)
 \right\rangle_{L^2(\omega^n)}
 =G_F^{\FR}(h_1,h_2).
\end{equation}
For $k=n$, this square-root representation is the mechanism behind
Calamai's spherical description of the Calabi metric \cite{Calamai2012}.

\subsection{The connection and curvature of the positive \texorpdfstring{$L^2$}{L2}-sphere}
\label{subsec:positive-sphere-connection}

In this subsection we recall the natural flat connection on the space of
smooth functions, its compatibility with the fixed $L^2$ metric, and the
Levi-Civita connection and curvature of the positive smooth $L^2$-sphere.
This also makes precise why the constant-curvature formula for a Hilbert
sphere continues to hold on the smooth positive locus, which is a
Fr\'echet rather than a Hilbert manifold.

Let $E:=\Cinf(X,\R)$, regarded as a Fr\'echet vector space.  Its linear
structure identifies $T_fE$ canonically with $E$ at every $f\in E$.
Equip $E$ with the fixed weak $L^2$ Riemannian metric
\begin{equation}\label{eq:flat-L2-metric}
 G^{L^2}_f(h_1,h_2)
 :=\langle h_1,h_2\rangle_2
 :=\int_X h_1h_2\,\omega^n,
 \qquad f\in E.
\end{equation}
  If $A,Y:E\to E$ are smooth vector fields, define
the natural flat connection by
\begin{equation}\label{eq:flat-D-definition}
 (D_A Y)_f:=dY_f(A(f))
  =\left.\frac{d}{dt}\right|_{t=0}Y(f+tA(f)).
\end{equation}
Here $dY_f$ is the derivative of a map between Fr\'echet spaces,
and the derivative on the right is taken in $E$, not on the base manifold
$X$.  In particular, $D_AY$ need not vanish even though the ambient
connection is flat.

We verify metric compatibility explicitly.  For smooth vector fields
$A,Y,Z$ on $E$, differentiation under the integral (using compactness
of $X$) and the product rule give
\begin{align}
 A\bigl(G^{L^2}(Y,Z)\bigr)\big|_f
 &=\left.\frac{d}{dt}\right|_{t=0}
   \int_X Y(f+tA(f))Z(f+tA(f))\,\omega^n \notag\\
 &=\int_X (D_AY)_f\,Z(f)\,\omega^n
   +\int_XY(f)\,(D_AZ)_f\,\omega^n \notag\\
 &=G^{L^2}_f((D_AY)_f,Z(f))
   +G^{L^2}_f(Y(f),(D_AZ)_f).
 \label{eq:flat-D-metric-compatible}
\end{align}
Thus $DG^{L^2}=0$.  For example, fix $h\in E$ and let
$A(f)=h$, $Y(f)=f^2$, and $Z(f)=f$.  Then
\[
 (D_AY)_f=2fh,\qquad (D_AZ)_f=h,
\]
so both sides of \eqref{eq:flat-D-metric-compatible} equal
\[
 3\int_X f^2h\,\omega^n
 =2\int_Xf^2h\,\omega^n+\int_Xf^2h\,\omega^n.
\]
The usual directional derivative on a vector space is also torsion-free
and flat:
\[
 D_AY-D_YA=[A,Y],\qquad
 D_AD_YZ-D_YD_AZ-D_{[A,Y]}Z=0.
\]
Consequently, $D$ is the Levi-Civita connection of this fixed weak
$L^2$ metric (its existence here follows from the explicit formula,
without invoking a general existence theorem for weak metrics).

Set $R:=2\sqrt V$ and write $S_R^+$ for the smooth positive sphere
\eqref{eq:positive-sphere}.  Since $X$ is compact, the condition $q>0$
is open in the smooth topology.  Differentiating the constraint
$\langle q,q\rangle_2=R^2$ gives
\begin{equation}\label{eq:positive-sphere-tangent}
 T_qS_R^+=\{h\in E:\langle q,h\rangle_2=0\}.
\end{equation}
The $L^2$-orthogonal projection from $E$ onto this smooth hyperplane is
\begin{equation}\label{eq:positive-sphere-projection}
 P_q(v)=v-\frac{\langle v,q\rangle_2}{R^2}q.
\end{equation}
For smooth tangent vector fields $A,Y$ on $S_R^+$, the ambient
derivative $D_AY$ is defined by differentiating $Y(q(t))$ along a curve
$q(t)$ in $S_R^+$ with $q'(0)=A(q)$.  Define the induced connection by
$\nabla_A Y:=P_q(D_A Y)$.  Since $\langle q,Y(q)\rangle_2=0$,
differentiation in the direction $A$ yields
$\langle q,D_A Y\rangle_2=-\langle A,Y\rangle_2$, and hence
\begin{equation}\label{eq:positive-sphere-LC}
 \boxed{\quad\nabla_A Y
  =D_A Y+\frac{\langle A,Y\rangle_2}{R^2}q.\quad}
\end{equation}
This expression remains smooth and tangent to $S_R^+$.  Projection of
the ambient metric-compatibility and torsion identities shows that
$\nabla$ is metric-compatible and torsion-free; it is therefore the
Levi-Civita connection of the induced weak metric on $S_R^+$.

\begin{proposition}[Curvature of the positive smooth sphere]
\label{prop:positive-sphere-curvature}
With the curvature convention
$\mathcal R(A,B)Z=\nabla_A\nabla_BZ-\nabla_B\nabla_AZ
-\nabla_{[A,B]}Z$, the induced connection on $S_R^+$ satisfies
\begin{equation}\label{eq:positive-sphere-curvature}
 \mathcal R(A,B)Z
 =\frac1{R^2}\bigl(\langle B,Z\rangle_2 A
                     -\langle A,Z\rangle_2 B\bigr).
\end{equation}
In particular, every two-plane spanned by linearly independent smooth
tangent vectors has sectional curvature
\[
 \operatorname{Sec}_{S_R^+}=\frac1{R^2}=\frac1{4V}.
\]
\end{proposition}

\begin{proof}
The position vector field on $S_R^+$ is $q\mapsto q$, so $D_Aq=A$.
Using \eqref{eq:positive-sphere-LC} and
$P_q(q)=0$, $P_q(A)=A$ for tangent $A$, we obtain
\begin{align*}
 \nabla_A\nabla_B Z
 &=P_q\!\left(D_A\left(D_BZ
      +\frac{\langle B,Z\rangle_2}{R^2}q\right)\right)\\
 &=P_q(D_AD_BZ)+\frac{\langle B,Z\rangle_2}{R^2}A.
\end{align*}
Interchanging $A$ and $B$ and subtracting
$\nabla_{[A,B]}Z=P_q(D_{[A,B]}Z)$, the ambient flatness of $D$
cancels the projected second-derivative terms and gives
\eqref{eq:positive-sphere-curvature}.  In particular,
\[
 \langle\mathcal R(A,B)B,A\rangle_2
 =\frac1{R^2}
   \bigl(\|A\|_2^2\|B\|_2^2-\langle A,B\rangle_2^2\bigr).
\]
Dividing by the Gram determinant in parentheses gives the stated
sectional curvature.  Thus the result is obtained within the smooth
Fr\'echet category.
\end{proof}

\section{The Hessian density map and the induced Fisher metric}

We now transfer the Fisher--Rao geometry of $\DV$ to the space of
$k$-Hessian potentials.

\subsection{The Hessian density map}

Define
\begin{equation}\label{eq:Tk-definition}
 \Tk:\Hkbar\longrightarrow\DV,
 \qquad
 \Tk([\varphi])=F_{\varphi}
 :=\frac{\omega_{\varphi}^{k}\wedge\omega^{n-k}}{\omega^n}.
\end{equation}
The map is well defined because the right-hand side is unchanged by adding a
constant to $\varphi$. Since $\varphi$ is $k$-admissible,
$F_{\varphi}>0$. Moreover,
\begin{equation}\label{eq:Tk-mass}
 \int_XF_{\varphi}\omega^n
 =\int_X\omega_{\varphi}^{k}\wedge\omega^{n-k}
 =\int_X\omega^n
 =V,
\end{equation}
where the middle equality follows from cohomology. Thus
$F_{\varphi}\in\DV$.

\begin{proposition}\label{prop:Tk-bijection}
The Hessian density map
\[
 \Tk:\Hkbar\longrightarrow\DV
\]
is a bijection.
\end{proposition}

\begin{proof}
Let $F\in\DV$. By Theorem~\ref{thm:hessian-solvability}, there exists a
smooth $k$-admissible function $\varphi$, unique modulo constants, satisfying
\[
 \omega_{\varphi}^{k}\wedge\omega^{n-k}=F\,\omega^n.
\]
Thus $\Tk([\varphi])=F$, proving surjectivity. Uniqueness modulo constants
gives injectivity.
\end{proof}

\subsection{Linearization of the Hessian density map}

Let $\varphi_s=\varphi+su$, where $u\in\Cinf(X,\R)$. Differentiating
\eqref{eq:Tk-definition} at $s=0$ gives
\begin{equation}\label{eq:dTk-explicit}
 D\Tk\big|_{\varphi}(u)
 =k\,
 \frac{\ddc u\wedge\omega_{\varphi}^{k-1}\wedge\omega^{n-k}}
 {\omega^n}.
\end{equation}
It is convenient to introduce the normalized linearized $k$-Hessian
operator
\begin{equation}\label{eq:k-Hessian-Laplacian}
 \Delta_{\varphi}^{(k)}u
 :=k\,
 \frac{\ddc u\wedge\omega_{\varphi}^{k-1}\wedge\omega^{n-k}}
 {\omega_{\varphi}^{k}\wedge\omega^{n-k}}.
\end{equation}
Then
\begin{equation}\label{eq:dTk-factorization}
 D\Tk\big|_{\varphi}(u)
 =F_{\varphi}\,\Delta_{\varphi}^{(k)}u.
\end{equation}
Since $\omega_{\varphi}\in\Gamma_k$, the corresponding Newton tensor is
positive definite; in particular, $\Delta_{\varphi}^{(k)}$ is elliptic.

\begin{proposition}\label{prop:dTk-isomorphism}
For every $\varphi\in\Hk$, the differential
\[
 D\Tk\big|_{\varphi}:
 \Cinf(X,\R)/\R
 \longrightarrow
 T_{F_{\varphi}}\DV
\]
is an isomorphism.
\end{proposition}

\begin{proof}
First, since
$\omega_{\varphi}^{k-1}\wedge\omega^{n-k}$ is closed,
\begin{equation}\label{eq:dTk-zero-mean}
 \int_XD\Tk\big|_{\varphi}(u)\,\omega^n
 =k\int_X\ddc u\wedge
 \omega_{\varphi}^{k-1}\wedge\omega^{n-k}
 =0.
\end{equation}
Thus
\[
 \operatorname{Im}D\Tk\big|_{\varphi}
 \subseteq T_{F_{\varphi}}\DV.
\]

Set
\[
 L_{\varphi}u:=D\Tk\big|_{\varphi}(u).
\]
For $u,v\in\Cinf(X,\R)$, integration by parts gives
\begin{align}
 \int_XvL_{\varphi}u\,\omega^n
 &=k\int_Xv\,\ddc u\wedge
   \omega_{\varphi}^{k-1}\wedge\omega^{n-k}\notag\\
 &=k\int_Xu\,\ddc v\wedge
   \omega_{\varphi}^{k-1}\wedge\omega^{n-k}\notag\\
 &=\int_XuL_{\varphi}v\,\omega^n.
 \label{eq:self-adjoint}
\end{align}
Hence $L_{\varphi}$ is formally self-adjoint with respect to the fixed
measure $\omega^n$. Ellipticity and the maximum principle imply
\begin{equation}\label{eq:L-kernel}
 \ker L_{\varphi}=\R.
\end{equation}
Working first in H\"older spaces, elliptic Fredholm theory together with
\eqref{eq:self-adjoint} yields
\[
 \operatorname{Im}L_{\varphi}
 =\left\{h\in C^{m,\alpha}(X):
 \int_Xh\,\omega^n=0\right\}.
\]
After quotienting the domain by constants, $L_{\varphi}$ is therefore an
isomorphism onto the tangent space in \eqref{eq:tangent-DV}. Elliptic
regularity gives the corresponding statement in the smooth category.
\end{proof}

\begin{remark}\label{rem:global-coordinate}
Propositions~\ref{prop:Tk-bijection} and \ref{prop:dTk-isomorphism} show that
the Hessian density provides a global coordinate on $\Hkbar$. If one wishes
to make the differential-manifold statement precise, one may first work in
H\"older completions, apply the Banach inverse function theorem, and then use
elliptic regularity to return to the smooth category.
\end{remark}

\subsection{The pulled-back Fisher metric}

We now define the metric of interest.

\begin{definition}\label{def:Hessian-Calabi-metric}
The \emph{$k$-Hessian Calabi metric} on $\Hkbar$ is the pullback of the
Fisher--Rao metric by the Hessian density map:
\begin{equation}\label{eq:pullback-definition}
 G^{\HC,k}_{\varphi}(u,v)
 :=G_{F_{\varphi}}^{\FR}
 \bigl(D\Tk|_{\varphi}(u),D\Tk|_{\varphi}(v)\bigr).
\end{equation}
\end{definition}

\begin{proposition}\label{prop:metric-formula}
For $u,v\in\Cinf(X,\R)/\R$,
\begin{equation}\label{eq:Hessian-Calabi-metric}
 G^{\HC,k}_{\varphi}(u,v)
 =\int_X
 \bigl(\Delta_{\varphi}^{(k)}u\bigr)
 \bigl(\Delta_{\varphi}^{(k)}v\bigr)
 \,\omega_{\varphi}^{k}\wedge\omega^{n-k}.
\end{equation}
In particular, $G^{\HC,k}$ is positive definite on $\Hkbar$.
\end{proposition}

\begin{proof}
Using \eqref{eq:fisher-rao} and \eqref{eq:dTk-factorization}, we obtain
\begin{align*}
 G^{\HC,k}_{\varphi}(u,v)
 &=\int_X
 \frac{F_{\varphi}^2
 \bigl(\Delta_{\varphi}^{(k)}u\bigr)
 \bigl(\Delta_{\varphi}^{(k)}v\bigr)}{F_{\varphi}}\,\omega^n\\
 &=\int_X
 \bigl(\Delta_{\varphi}^{(k)}u\bigr)
 \bigl(\Delta_{\varphi}^{(k)}v\bigr)
 \,\omega_{\varphi}^{k}\wedge\omega^{n-k}.
\end{align*}
If $G^{\HC,k}_{\varphi}(u,u)=0$, then
$\Delta_{\varphi}^{(k)}u=0$. By ellipticity and the maximum principle,
$u$ is constant, hence zero in the quotient tangent space.
\end{proof}

\begin{remark}\label{rem:k-equals-n}
When $k=n$, the operator \eqref{eq:k-Hessian-Laplacian} is the ordinary
K\"ahler Laplacian, up to the convention used for $dd^c$, and
\eqref{eq:Hessian-Calabi-metric} becomes
\[
 G^{\HC,n}_{\varphi}(u,v)
 =\int_X
 (\Delta_{\omega_{\varphi}}u)
 (\Delta_{\omega_{\varphi}}v)
 \,\omega_{\varphi}^n,
\]
which is the classical Calabi metric studied by Calamai
\cite{Calamai2012}.
\end{remark}

\subsection{The square-root Hessian map}

Define
\begin{equation}\label{eq:Qk-definition}
 \Qk:\Hkbar\longrightarrow L^2(X,\omega^n),
 \qquad
 \Qk([\varphi])
 :=2\sqrt{F_{\varphi}}
 =2\sqrt{\frac{\omega_{\varphi}^{k}\wedge\omega^{n-k}}{\omega^n}}.
\end{equation}
By \eqref{eq:Tk-mass},
\begin{equation}\label{eq:Qk-norm}
 \|\Qk([\varphi])\|_{L^2(\omega^n)}^2
 =4\int_XF_{\varphi}\omega^n
 =4V.
\end{equation}
Moreover, $\Qk([\varphi])>0$. Hence
\[
 \Qk(\Hkbar)\subseteq S_{2\sqrt V}^{+}.
\]
Differentiating \eqref{eq:Qk-definition} gives
\begin{equation}\label{eq:dQk}
 D\Qk\big|_{\varphi}(u)
 =\frac{D\Tk|_{\varphi}(u)}{\sqrt{F_{\varphi}}}
 =\sqrt{F_{\varphi}}\,\Delta_{\varphi}^{(k)}u.
\end{equation}
Consequently,
\begin{align}
 \left\langle D\Qk|_{\varphi}(u),D\Qk|_{\varphi}(v)
 \right\rangle_{L^2(\omega^n)}
 &=\int_XF_{\varphi}
 \bigl(\Delta_{\varphi}^{(k)}u\bigr)
 \bigl(\Delta_{\varphi}^{(k)}v\bigr)\,\omega^n\notag\\
 &=G^{\HC,k}_{\varphi}(u,v).
 \label{eq:Qk-local-isometry}
\end{align}

\begin{theorem}\label{thm:global-isometry}
The square-root Hessian density map
\[
 \Qk:
 (\Hkbar,G^{\HC,k})
 \longrightarrow
 S_{2\sqrt V}^{+}\subset L^2(X,\omega^n)
\]
is a global Riemannian isometry.
\end{theorem}

\begin{proof}
Equation \eqref{eq:Qk-local-isometry} shows that $\Qk$ preserves the
Riemannian metric. On the other hand,
\[
 \Qk=\mathcal S\circ\Tk.
\]
The map $\Tk$ is a bijection by Proposition~\ref{prop:Tk-bijection}, while
$\mathcal S(F)=2\sqrt F$ is a bijection from $\DV$ onto
$S_{2\sqrt V}^{+}$. Hence $\Qk$ is bijective and therefore a global
Riemannian isometry.
\end{proof}

\begin{remark}\label{rem:comparison-GST}
Theorem~\ref{thm:global-isometry} gives a useful contrast with the
Hessian--Mabuchi geometry of Guo--Smith--Tong. Their metric
\eqref{eq:GST-metric} is a zeroth-order $L^2$ metric built directly from the
Hessian measure, whereas \eqref{eq:Hessian-Calabi-metric} is obtained by
pulling back the Fisher--Rao metric under the Hessian density map. Thus the
former is naturally viewed as a $k$-Hessian analogue of the Mabuchi metric,
while the latter is a $k$-Hessian analogue of the Calabi metric. When $k=n$,
Theorem~\ref{thm:global-isometry} reduces to the spherical model of Calamai
\cite{Calamai2012}. Consequently, the geodesic, distance, curvature and
metric-completion questions for $G^{\HC,k}$ can be transferred to the
geometry of the positive $L^2$-sphere.
\end{remark}

\section{Some Corollaries}
\label{sec:some-applications}

Throughout this section, we use the global isometry in
Theorem~\ref{thm:global-isometry}. Set
\[
 R:=2\sqrt V,
 \qquad
 q_{\varphi}:=\Qk([\varphi])=2\sqrt{F_{\varphi}}.
\]
Thus $q_{\varphi}$ belongs to the positive part of the $L^2$-sphere of
radius $R$. The geometric consequences below follow from the corresponding
facts for a round sphere, with the qualification that the geodesic must
remain in the pointwise-positive locus.

\begin{corollary}[Constant sectional curvature]
\label{cor:constant-sectional-curvature}
The metric $G^{\HC,k}$ on $\Hkbar$ has constant sectional curvature
\[
 \operatorname{Sec}_{G^{\HC,k}}=\frac{1}{4V}.
\]
\end{corollary}

\begin{proof}
By Theorem~\ref{thm:global-isometry},
$(\Hkbar,G^{\HC,k})$ is isometric to the positive smooth
$L^2(X,\omega^n)$-sphere of radius $R=2\sqrt V$.
Proposition~\ref{prop:positive-sphere-curvature} computes the
sectional curvature of its induced weak $L^2$ metric as $R^{-2}$; hence
\[
 R^{-2}=(2\sqrt V)^{-2}=\frac{1}{4V}.
\]
\end{proof}

\begin{corollary}[The geodesic initial value problem]
\label{cor:geodesic-IVP}
Fix $[\varphi_0]\in\Hkbar$ and
$u_0\in T_{[\varphi_0]}\Hkbar$. Put
\[
 q_0:=\Qk([\varphi_0]),\qquad
 h_0:=D\Qk|_{\varphi_0}(u_0),\qquad
 s:=\|h_0\|_{L^2(\omega^n)}
   =\sqrt{G^{\HC,k}_{\varphi_0}(u_0,u_0)}.
\]
If $s>0$, define
\begin{equation}\label{eq:geodesic-sphere-IVP}
 q(t)=
 \cos\!\left(\frac{s}{R}t\right)q_0
 +\frac{R}{s}\sin\!\left(\frac{s}{R}t\right)h_0,
 \qquad R=2\sqrt V.
\end{equation}
If $s=0$, set $q(t)=q_0$. The unique geodesic with initial data
$[\varphi(0)]=[\varphi_0]$ and $\dot{[\varphi]}(0)=u_0$ is given by
\[
 [\varphi(t)]=\Qk^{-1}(q(t))
\]
on the maximal connected time interval containing $0$ on which
$q(t,x)>0$ for every $x\in X$.

In particular, if $s>0$, this geodesic cannot exist for all
$t\in\mathbb R$.
\end{corollary}

\begin{proof}
The differential of $\Qk$ is an isometry by
\eqref{eq:Qk-local-isometry}. In particular,
\[
 \langle q_0,h_0\rangle_{L^2(\omega^n)}=0,
 \qquad \|q_0\|_{L^2(\omega^n)}=R,
 \qquad \|h_0\|_{L^2(\omega^n)}=s.
\]
Formula~\eqref{eq:geodesic-sphere-IVP} is the standard great-circle solution
on the radius-$R$ sphere with initial point $q_0$ and initial velocity
$h_0$. Whenever $q(t)$ remains strictly positive, it lies in
$S_R^+$, and the global isometry gives the stated geodesic in $\Hkbar$.
Conversely, any geodesic in $\Hkbar$ is carried by $\Qk$ to this spherical
geodesic, which proves uniqueness and the maximal-interval statement.

For the final assertion, suppose $s>0$. Since $q_0>0$ and
$\langle q_0,h_0\rangle_{L^2}=0$, the nonzero function $h_0$ must take both
positive and negative values. At
\[
 t_*=\frac{\pi R}{2s},
\]
formula~\eqref{eq:geodesic-sphere-IVP} gives
\[
 q(t_*)=\frac{R}{s}h_0,
\]
which is not strictly positive. Hence the spherical geodesic leaves the
positive locus no later than $t_*$, and the corresponding geodesic in
$\Hkbar$ cannot be extended for all positive time.
\end{proof}

\begin{corollary}[Distance formula and diameter]
\label{cor:distance-diameter}
For any $[\varphi_0],[\varphi_1]\in\Hkbar$, write
$F_i=F_{\varphi_i}$ for $i=0,1$. Then
\begin{equation}\label{eq:HC-distance-formula}
 d_{\HC,k}([\varphi_0],[\varphi_1])
 =2\sqrt V\,
 \arccos\!\left(
 \frac{1}{V}\int_X\sqrt{F_0F_1}\,\omega^n
 \right).
\end{equation}
Consequently,
\begin{equation}\label{eq:HC-diameter}
 \operatorname{diam}(\Hkbar,d_{\HC,k})=\pi\sqrt V.
\end{equation}
The diameter is a supremum and is not attained by any pair of points.
\end{corollary}

\begin{proof}
Set $q_i=\Qk([\varphi_i])=2\sqrt{F_i}$. Since $q_0,q_1$ are strictly
positive and have $L^2$-norm $R=2\sqrt V$, their spherical angle
\[
 \theta=\arccos\!\left(
 \frac{\langle q_0,q_1\rangle_{L^2(\omega^n)}}{R^2}
 \right)
 =\arccos\!\left(
 \frac{1}{V}\int_X\sqrt{F_0F_1}\,\omega^n
 \right)
\]
satisfies $0\leq\theta<\pi/2$. The shorter great-circle arc joining
$q_0$ to $q_1$ stays strictly positive: for $0\leq t\leq1$, it is
\[
 q_t=
 \frac{\sin((1-t)\theta)}{\sin\theta}\,q_0
 +\frac{\sin(t\theta)}{\sin\theta}\,q_1
\]
when $\theta>0$, with the constant path understood when $\theta=0$.
Thus the intrinsic distance in $S_R^+$ equals the spherical distance
$R\theta$, which proves \eqref{eq:HC-distance-formula}.

Strict positivity implies
$\int_X\sqrt{F_0F_1}\,\omega^n>0$, so every pair satisfies
\[
 d_{\HC,k}([\varphi_0],[\varphi_1])<\frac{\pi R}{2}
 =\pi\sqrt V.
\]
On the other hand, one can choose pairs of smooth strictly positive
densities of mass $V$ whose square roots become concentrated in
essentially disjoint regions, so that
$\int_X\sqrt{F_0F_1}\,\omega^n\to0$. For example, start with nonnegative
smooth functions supported in disjoint open sets, add a positive constant
tending to zero, and normalize their integrals to $V$. By
Proposition~\ref{prop:Tk-bijection}, each such density is the Hessian
density of a point of $\Hkbar$. Their distances therefore approach
$\pi R/2=\pi\sqrt V$. This proves \eqref{eq:HC-diameter}; the strict
inequality for each pair shows that the diameter is not attained.
\end{proof}

\section{Comparison of adjacent Hessian--Fisher distances}
\label{sec:adjacent-distances}

The preceding construction places the Hessian density maps of different
orders in the same Fisher space. We now compare the corresponding
pullback distances on a $C^2$-small neighborhood of the reference K\"ahler
metric. Unlike a comparison of path lengths along a prescribed curve, the
result below concerns the original geodesic distances on the full spaces
of admissible potentials.

For $1\leq j\leq n$, write
\begin{equation}\label{eq:comparison-FQ}
 F_j(\varphi)
 :=\frac{\omega_{\varphi}^{j}\wedge\omega^{n-j}}{\omega^n},
 \qquad
 Q_j([\varphi]):=2\sqrt{F_j(\varphi)},
\end{equation}
and let $d_j:=d_{\HC,j}$ denote the geodesic distance induced by
$G^{\HC,j}$ on $\widetilde{\mathcal H}_j(X,\omega)$. All these maps take
values in the same positive $L^2$-sphere of radius $R=2\sqrt V$. For
$\delta>0$, define
\begin{equation}\label{eq:comparison-Udelta}
 \mathcal U_\delta
 :=\left\{[\varphi]\in\widetilde{\mathcal H}_n(X,\omega):
 \frac{1}{1+\delta}\,\omega
 <\omega_{\varphi}<
 (1+\delta)\,\omega\right\}.
\end{equation}
Here inequalities between real $(1,1)$-forms are in the sense of positive
definiteness. In particular, $\mathcal U_\delta$ is convex under affine
interpolation of representatives and is contained in every admissible
potential space.

\begin{theorem}[Local comparison of adjacent Hessian--Fisher--Rao distances]
\label{thm:adjacent-distance-comparison}
Let $(X,\omega)$ be a compact K\"ahler manifold of complex dimension $n$,
and let $1\leq k<k+1\leq n$. There are constants
$\delta_0=\delta_0(n,k)>0$ and $C=C(n,k)>0$ such that, for
$0<\delta<\delta_0$ and any distinct
$[\varphi_0],[\varphi_1]\in\mathcal U_\delta$, one has
\begin{equation}\label{eq:comparison-main-ratio}
 \left|
 \frac{d_{k+1}([\varphi_0],[\varphi_1])}
      {d_k([\varphi_0],[\varphi_1])}
 -\frac{k+1}{k}\right|\leq C\delta.
\end{equation}
In particular, after possibly decreasing $\delta_0$,
\begin{equation}\label{eq:comparison-strict}
 d_{k+1}([\varphi_0],[\varphi_1])
 >d_k([\varphi_0],[\varphi_1]).
\end{equation}
The distances here are those on the full spaces
$\widetilde{\mathcal H}_{k+1}$ and $\widetilde{\mathcal H}_k$,
respectively, evaluated at endpoints in $\mathcal U_\delta$; paths are
\emph{not} required to stay in $\mathcal U_\delta$.
\end{theorem}

\begin{proof}
Write $\|\cdot\|_2$ for the $L^2(X,\omega^n)$ norm, and use the
background K\"ahler metric $\omega$ to measure tensors. The constants in
$O(\delta)$ below are uniform over the potentials, tangent vectors and
endpoints under consideration, and depend only on $n$ and the indicated
Hessian order.

\medskip\noindent\textbf{Step 1. A uniform estimate for the differentials.}
We first show that, when $\omega_\varphi$ is sufficiently close to $\omega$,
the differential of the square-root Hessian density map $Q_j$ is a
uniformly small perturbation of $(j/n)\Delta_\omega$.
Fix $x\in X$ and choose $\omega$-unitary complex coordinates at $x$
such that the Hermitian endomorphism
$A:=\omega^{-1}\omega_\varphi$ is diagonal:
\[
 A=\operatorname{diag}(\lambda_1,\ldots,\lambda_n).
\]
For a smooth real tangent vector $u$, write
$B:=\omega^{-1}\ddc u=(B_{i\bar l})$. Thus
$\Delta_\omega u:=\tr_\omega(\ddc u)=\sum_i B_{i\bar i}$.
In these coordinates $\omega_{\varphi+tu}$ corresponds to $A+tB$.

For $j\in\{k,k+1\}$, the normalized $j$-Hessian density is
\[
 F_j(\varphi)=\frac{\sigma_j(\lambda)}{\binom nj},
\qquad
 \sigma_j(\lambda)
 =\sum_{1\le i_1<\cdots<i_j\le n}
   \lambda_{i_1}\cdots\lambda_{i_j}.
\]
To differentiate $F_j$ in the direction $u$, recall that
$\sigma_j(A)$ is the sum of all principal $j\times j$ minors of $A$.
At a diagonal matrix $A$, the first variation of such a minor in the
direction $B$ consists only of its diagonal variations: changing an
off-diagonal entry contributes no linear term. More explicitly,
\[
 \left.\frac{d}{dt}\right|_{t=0}\sigma_j(A+tB)
 =\sum_{|I|=j}\sum_{i\in I}
   \left(\prod_{l\in I\setminus\{i\}}\lambda_l\right)
   B_{i\bar i}
 =\sum_{i=1}^n\sigma_{j-1}(\lambda|i)B_{i\bar i},
\]
where $\lambda|i$ denotes the list of eigenvalues with $\lambda_i$
omitted. Equivalently,
$\partial\sigma_j/\partial\lambda_i=\sigma_{j-1}(\lambda|i)$.
Consequently,
\begin{equation}\label{eq:comparison-dFj}
 F_j(\varphi)=\frac{\sigma_j(\lambda)}{\binom nj},
 \qquad
 DF_j\big|_\varphi(u)
 =\frac{1}{\binom nj}
  \sum_{i=1}^n\sigma_{j-1}(\lambda|i)B_{i\bar i}.
\end{equation}
Since $Q_j=2\sqrt{F_j}$, the chain rule gives
$DQ_j|_\varphi(u)=DF_j|_\varphi(u)/\sqrt{F_j(\varphi)}$; hence
\begin{equation}\label{eq:comparison-dQj}
 DQ_j\big|_\varphi(u)
 =\sum_{i=1}^n
 \frac{\sigma_{j-1}(\lambda|i)}
 {\binom nj\sqrt{F_j(\varphi)}}\,B_{i\bar i}.
\end{equation}

Let us denote the coefficient of $B_{i\bar i}$ by
\[
 a_i^{(j)}(\lambda)
 :=\frac{\sigma_{j-1}(\lambda|i)}
 {\binom nj\sqrt{\sigma_j(\lambda)/\binom nj}}.
\]
At the reference vector ${\bf 1}=(1,\ldots,1)$ one has
$\sigma_j({\bf 1})=\binom nj$ and
$\sigma_{j-1}({\bf 1}|i)=\binom{n-1}{j-1}$. Therefore
\[
 a_i^{(j)}({\bf 1})
 =\frac{\binom{n-1}{j-1}}{\binom nj}
 =\frac jn,
 \qquad
 DQ_j\big|_0(u)=\frac jn\Delta_\omega u.
\]

For $[\varphi]\in\mathcal U_\delta$,
\[
 (1+\delta)^{-1}<\lambda_i<1+\delta,
 \qquad |\lambda_i-1|\le\delta.
\]
Take, for example, $0<\delta<1/2$. The eigenvalue vector then lies
in the fixed compact cube $[1/2,2]^n$, where each
$a_i^{(j)}$ is smooth and all its first partial derivatives are bounded
by a constant depending only on $n,j$. Applying the multivariable mean
value theorem along the segment from ${\bf 1}$ to $\lambda$ yields
\[
 \left|a_i^{(j)}(\lambda)-\frac jn\right|
 \le C_{n,j}\max_l|\lambda_l-1|
 \le C_{n,j}\delta.
\]
Subtracting $(j/n)\Delta_\omega u$ from
\eqref{eq:comparison-dQj} and applying Cauchy--Schwarz,
\begin{align*}
 \left|DQ_j\big|_\varphi(u)-\frac jn\Delta_\omega u\right|
 &\le\sum_{i=1}^n
  \left|a_i^{(j)}(\lambda)-\frac jn\right|\,|B_{i\bar i}|\\
 &\le C_{n,j}\delta\sum_{i=1}^n|B_{i\bar i}|\\
 &\le C_{n,j}\delta\,|B|_\omega.
\end{align*}
Here the harmless factor $\sqrt n$ has been absorbed into $C_{n,j}$,
and $|B|_\omega=|\ddc u|_\omega$. We conclude that
\begin{equation}\label{eq:comparison-pointwise-error}
 \left|DQ_j\big|_\varphi(u)-\frac jn\Delta_\omega u\right|
 \leq C_{n,j}\delta\,|\ddc u|_\omega.
\end{equation}
This is a pointwise statement. We do not claim that
$|\ddc u|_\omega$ is pointwise controlled by $|\Delta_\omega u|$;
that additional control holds after integration and is the purpose of
Step~2.

\medskip\noindent\textbf{Step 2. The K\"ahler $L^2$ identity.}
For a real smooth function $u$, the pointwise algebraic identity for real
$(1,1)$-forms reads
\[
 \bigl((\Delta_\omega u)^2-|\ddc u|_\omega^2\bigr)\omega^n
 =n(n-1)(\ddc u)^2\wedge\omega^{n-2}.
\]
The integral on the right vanishes by Stokes' theorem, because
$\omega$ is closed and $\ddc u$ is exact. Thus
\begin{equation}\label{eq:comparison-kahler-L2}
 \|\ddc u\|_2=\|\Delta_\omega u\|_2.
\end{equation}
Combining \eqref{eq:comparison-pointwise-error} with
\eqref{eq:comparison-kahler-L2}, we obtain
\begin{equation}\label{eq:comparison-differential-L2}
 \left\|DQ_j\big|_\varphi(u)-\frac jn\Delta_\omega u\right\|_2
 \leq C_{n,j}\delta\,\|\Delta_\omega u\|_2,
 \qquad [\varphi]\in\mathcal U_\delta.
\end{equation}
In particular, the same estimate shows at the tensor level that
\begin{equation}\label{eq:comparison-tangent-estimate}
 \left(\frac jn-C_{n,j}\delta\right)\|\Delta_\omega u\|_2
 \leq\sqrt{G^{\HC,j}_\varphi(u,u)}
 \leq
 \left(\frac jn+C_{n,j}\delta\right)\|\Delta_\omega u\|_2.
\end{equation}

\medskip\noindent\textbf{Step 3. Comparison of the endpoint chord lengths.}
Choose representatives and set $u=\varphi_1-\varphi_0$.
The affine segment
$\varphi_t=(1-t)\varphi_0+t\varphi_1$ belongs to
$\mathcal U_\delta$, since
$\omega_{\varphi_t}=(1-t)\omega_{\varphi_0}+t\omega_{\varphi_1}$.
By the fundamental theorem of calculus and
\eqref{eq:comparison-differential-L2},
\begin{align}
 Q_j([\varphi_1])-Q_j([\varphi_0])
 &=\int_0^1 DQ_j\big|_{\varphi_t}(u)\,dt,\notag\\
 \left\|Q_j([\varphi_1])-Q_j([\varphi_0])
 -\frac jn\Delta_\omega u\right\|_2
 &\leq C_{n,j}\delta\,\|\Delta_\omega u\|_2.
 \label{eq:comparison-chord-error}
\end{align}
Let
\[
 L:=\|\Delta_\omega u\|_2,
 \qquad
 \ell_j:=\|Q_j([\varphi_1])-Q_j([\varphi_0])\|_2.
\]
Since the endpoints are distinct modulo constants, $L>0$. The triangle
inequality in \eqref{eq:comparison-chord-error} gives
\begin{equation}\label{eq:comparison-chord-bounds}
 \left(\frac jn-C_{n,j}\delta\right)L
 \leq \ell_j\leq
 \left(\frac jn+C_{n,j}\delta\right)L.
\end{equation}

\medskip\noindent\textbf{Step 4. From chords to geodesic distances.}
The preceding step controls the $L^2$ chord length $\ell_j$ between the
images of the endpoints. We now use the global spherical description to
convert this estimate into one for the original geodesic distance $d_j$.

First, the condition $[\varphi]\in\mathcal U_\delta$ implies
$(1+\delta)^{-1}<\lambda_i<1+\delta$. Since $F_j$ is the average of
all products of $j$ distinct eigenvalues, it follows that
\[
 (1+\delta)^{-j}<F_j(\varphi)<(1+\delta)^j.
\]
For $\delta$ sufficiently small these bounds give the uniform
\emph{pointwise} estimate
\[
 |F_j(\varphi)-1|\le C_{n,j}\delta.
\]
Using $Q_j=2\sqrt{F_j}$ and the identity
$|\sqrt a-1|=|a-1|/(\sqrt a+1)$ for $a>0$, we obtain
\[
 |Q_j([\varphi])(x)-2|
 =\frac{2|F_j(\varphi)(x)-1|}{\sqrt{F_j(\varphi)(x)}+1}
 \le C_{n,j}\delta,
 \qquad x\in X.
\]
The constant function $2$ lies on the radius-$R$ sphere, since
$\|2\|_2^2=4V=R^2$. Put
$q_i=Q_j([\varphi_i])$ for $i=0,1$. Then
\[
 \|q_i-2\|_2\le C_{n,j}\delta\sqrt V,
\]
and hence the triangle inequality gives
\[
 \ell_j=\|q_1-q_0\|_2
 \le\|q_1-2\|_2+\|q_0-2\|_2
 \le 2C_{n,j}\delta\sqrt V.
\]
As $R=2\sqrt V$, this proves
\begin{equation}\label{eq:comparison-short-chords}
 \frac{\ell_j}{R}\leq C_{n,j}\delta.
\end{equation}

Next, recall from Theorem~\ref{thm:global-isometry} and
Corollary~\ref{cor:distance-diameter} that the ambient distance $d_j$
on the full space $\widetilde{\mathcal H}_j$ is the intrinsic distance
between $q_0,q_1$ in the positive part of the radius-$R$ sphere.
In particular, the shorter great-circle arc is entirely composed of
positive functions and is minimizing; hence if
\[
 \theta_j
 :=\arccos\!\left(\frac{\langle q_0,q_1\rangle_{L^2}}{R^2}\right)
 \in[0,\pi/2),
\]
then
\[
 d_j([\varphi_0],[\varphi_1])=R\theta_j.
\]
Although the ambient Hilbert space is infinite-dimensional, the two
points and the origin span a subspace of dimension at most two, so the
usual elementary circle geometry applies. Indeed,
\begin{align*}
 \ell_j^2
 &=\|q_1-q_0\|_2^2
  =\|q_1\|_2^2+\|q_0\|_2^2-2\langle q_0,q_1\rangle_{L^2}\\
 &=2R^2(1-\cos\theta_j)
  =4R^2\sin^2(\theta_j/2).
\end{align*}
Since $0\le\theta_j<\pi/2$, we have
$\ell_j=2R\sin(\theta_j/2)$, which is equivalent to the exact
chord--arc formula
\begin{equation}\label{eq:comparison-chord-arc}
 d_j([\varphi_0],[\varphi_1])
 =2R\arcsin\!\left(\frac{\ell_j}{2R}\right).
\end{equation}

By \eqref{eq:comparison-short-chords}, the argument
$z_j:=\ell_j/(2R)$ satisfies $0\le z_j\le C_{n,j}\delta$.
For $\delta$ small enough that $z_j$ stays in a fixed neighborhood of
zero, Taylor's formula yields uniformly
\[
 \arcsin z=z+\frac{z^3}{6}+O(z^5).
\]
Substituting $z_j$ in \eqref{eq:comparison-chord-arc} gives
\[
 d_j
 =\ell_j+\frac{\ell_j^3}{24R^2}
   +O\!\left(\frac{\ell_j^5}{R^4}\right)
 =\ell_j\left(1+O_{n,j}\!\left(\frac{\ell_j^2}{R^2}\right)\right).
\]
In view of \eqref{eq:comparison-short-chords}, the relative error is
$O_{n,j}(\delta^2)$. Consequently,
\begin{equation}\label{eq:comparison-distance-chord}
 d_j([\varphi_0],[\varphi_1])
 =\ell_j\bigl(1+O_{n,j}(\delta^2)\bigr).
\end{equation}
Notice that the conversion from chord length to spherical arc length
introduces an error of order $\delta^2$, whereas the earlier estimate
\eqref{eq:comparison-chord-bounds} has an error of order $\delta$.
Combining these two estimates, with
$L=\|\Delta_\omega(\varphi_1-\varphi_0)\|_2$, yields
\begin{equation}\label{eq:comparison-distance-Laplacian}
 d_j([\varphi_0],[\varphi_1])
 =\left(\frac jn+O_{n,j}(\delta)\right)
 \|\Delta_\omega(\varphi_1-\varphi_0)\|_2,
 \qquad j=k,k+1.
\end{equation}
Here the notation means two-sided bounds with an error of absolute value
at most $C_{n,j}\delta$ times the indicated positive norm; it does not
assume any differentiability of distance along the affine segment.

\medskip\noindent\textbf{Step 5. The ratio and strict monotonicity.}
Dividing the two estimates in
\eqref{eq:comparison-distance-Laplacian}, with $\delta$ small enough so
that the denominator is bounded away from zero, gives
\[
 \frac{d_{k+1}([\varphi_0],[\varphi_1])}
      {d_k([\varphi_0],[\varphi_1])}
 =\frac{(k+1)/n+O_{n,k}(\delta)}
       {k/n+O_{n,k}(\delta)}
 =\frac{k+1}{k}+O_{n,k}(\delta),
\]
which proves \eqref{eq:comparison-main-ratio}. Finally,
$(k+1)/k>1$, so we can choose $\delta_0(n,k)>0$ small enough that the
right-hand side of \eqref{eq:comparison-main-ratio} is strictly greater
than $1$. This proves \eqref{eq:comparison-strict}.
\end{proof}

\begin{remark}\label{rem:comparison-intrinsic}
One can also consider the intrinsic length distances obtained by
restricting the paths in the definition of $d_j$ to $\mathcal U_\delta$.
For these \emph{restricted} distances, the uniform tangent-space bounds
\eqref{eq:comparison-tangent-estimate} directly compare the lengths of
every admissible path, and hence give the same ratio
$(k+1)/k+O_{n,k}(\delta)$ after taking infima. Theorem~\ref{thm:adjacent-distance-comparison}
is stronger in a different sense: it compares the original ambient
distances, for which a minimizing geodesic need not remain inside
$\mathcal U_\delta$.
\end{remark}


\vskip10pt
Linwei Zhang, School of Mathematical Sciences,
University of Science and Technology of China, No. 96 Jinzhai Road, Hefei,
Anhui Province, 230026, China;  zhanglinwei@mail.ustc.edu.cn.

\end{document}